\documentclass[11pt]{article}

\usepackage{amsmath,amssymb,amsfonts,amsthm}
\usepackage{thmtools}
\usepackage{mathrsfs}
\usepackage{geometry}
\usepackage{xcolor}
\usepackage{hyperref}
\usepackage[capitalise,nameinlink,noabbrev]{cleveref}
\usepackage{enumitem}
\usepackage{tcolorbox}

\hypersetup{
    colorlinks=true,
    linkcolor={red!50!black},
    citecolor={blue!50!black},
    urlcolor={magenta!70!black}
}

\numberwithin{equation}{section}

\newtheorem{theorem}{Theorem}[section]
\newtheorem{lemma}[theorem]{Lemma}
\newtheorem{corollary}[theorem]{Corollary}
\newtheorem{proposition}[theorem]{Proposition}

\theoremstyle{definition}
\newtheorem{definition}[theorem]{Definition}
\newtheorem{remark}[theorem]{Remark}

\newcommand{\R}{\mathbb{R}}
\newcommand{\C}{\mathbb{C}}
\newcommand{\Hh}{\mathbb{H}}

\newcommand{\dd}{\,\mathrm{d}}

\tcbset{
  violetbox/.style={
    colback=violet!6!white,
    colframe=violet!50!black,
    colbacktitle=violet!65!black,
    coltitle=white,
    fonttitle=\bfseries
  }
}

\title{A note on harmonic polynomials on Heisenberg and Carnot groups}
\author{Francesco Paolo Maiale}
\date{}

\begin{document}
\maketitle

\begin{abstract}
We study harmonic polynomial traces on Kor\'anyi spheres in the Heisenberg
groups and polynomial trace filtrations on general Carnot groups. The
homogeneous harmonic trace spaces are algebraically direct and complete on
the Heisenberg sphere, but different homogeneous degrees are not generally
orthogonal. In fact, on every $\Hh^n$ the degree-one and degree-three
spaces fail to be orthogonal for every finite positive full-support surface
measure.
The decomposition thus uses the orthogonal increments of the
cumulative harmonic-degree filtration. These increments give a complete
Hilbert sum, have dimension $\binom{m+2n-1}{2n-1}$, and retain a
$U(n)$-equivariant refinement. A constructive triangular recursion produces
all homogeneous harmonic polynomials and their dimensions.

On an arbitrary Carnot group, the same harmonic filtration always
decomposes its closed span; equality with the full spherical $L^2$ space is
identified as a separate harmonic-density condition. By contrast, the
filtration by all polynomial traces is unconditionally complete by
Stone--Weierstrass.  We describe its dimensions through the vanishing ideal
of the gauge sphere and compute them explicitly for every $\Hh^n$.

Our main algebraic result is a Fischer-type decomposition. With $\eta_+^2(z,t)=|z|^2+4t$, we prove that $P_m(\mathbb H)=H_m(\mathbb H)\oplus \eta_+^2P_{m-2}(\mathbb H)$.
\end{abstract}



\section{Introduction}

The aim of this paper is to develop a concrete picture of spherical harmonic
analysis for the sub-Laplacian on the Heisenberg group and, more generally,
on Carnot groups.

On Euclidean space $\mathbb R^d$, homogeneous harmonic polynomials of degree
$m$ restrict to spherical harmonics on $\mathbb S^{d-1}$, and
$L^2(\mathbb S^{d-1})$ is their orthogonal Hilbert direct sum; see, for
example, Stein--Weiss~\cite[Ch.~IV]{SteinWeiss1971} or
Axler--Bourdon--Ramey~\cite[Ch.~5]{AxlerBourdonRamey}. This
statement combines three logically different properties: algebraic
directness, completeness, and orthogonality.

On Heisenberg and Carnot groups, a parallel theory is available through the
work of Folland--Stein~\cite{FollandStein1974,FollandStein1982},
Greiner~\cite{Greiner1980}, and subsequent authors. Our paper keeps
a polynomial and PDE viewpoint, but distinguishes carefully
between the three properties above. On Kor\'anyi spheres the homogeneous
harmonic traces are algebraically direct and complete, whereas different
homogeneous degrees are not generally orthogonal.

\subsection{Heisenberg groups and the sub-Laplacian operator}
\label{sec:heisenberg-intro}

For orientation, identify the first Heisenberg group
$\Hh=\Hh^1$ with $\C\times\R$ and use the group law
\[
  (z,t)(z',t')
  =\bigl(z+z',t+t'+2\operatorname{Im}(z\overline{z'})\bigr).
\]
Writing $z=x+\imath y$, with $x,y \in \R$, the horizontal left-invariant fields and the
analytic-sign sub-Laplacian are
\[
  X=\partial_x+2y\partial_t,
  \qquad
  Y=\partial_y-2x\partial_t,
  \qquad
  \Delta_H=X^2+Y^2.
\]
Equivalently, with $T=\partial_t$, the complex horizontal fields are
\[
  Z=\frac12(X-\imath Y)
    =\partial_z+\imath \overline z\,\partial_t,
  \qquad
  \overline Z=\frac12(X+\imath Y)
    =\partial_{\overline z}-\imath z\,\partial_t.
\]
The commutator $[X,Y]=-4\partial_t$ generates the central direction, so
$\Delta_H$ is hypoelliptic but not elliptic.

In~\cref{sec:prelim}, we
describe in details the full higher-dimensional group law, the real and complex vector
fields, the stratification, the expanded formula for $\Delta_H$, and the
Kor\'anyi polar measure.

\subsection{Homogeneity and statements of the main results}

We now fix the homogeneity convention. Recall that the central variable has 
weight two under the Heisenberg dilations:
\[
  \delta_r(z,t)=(rz,r^2t),\qquad r>0.
\]

\begin{definition}\label{def:homogeneous}
Let $n\geq1$.  A polynomial $p$ on $\Hh^n$ is
\emph{homogeneous of Heisenberg degree $m$} if
\[
  p\circ\delta_r=r^m p
  \qquad\text{for every }r>0.
\]
With complex coefficients, we write
\[
  P_m(\Hh^n)
   = \bigl\{p \ : \ p\circ\delta_r=r^mp \bigr\},
  \qquad
  H_m(\Hh^n)
  = \ker\bigl(\Delta_H|_{P_m(\Hh^n)}\bigr).
\]
We use the convention $P_j(\Hh^n)=\{0\}$ for $j<0$. When $n=1$ we write
$P_m(\Hh)$ and $H_m(\Hh)$.
\end{definition}

Following Folland--Stein, Greiner introduced in~\cite{Greiner1980} a one-parameter family of
left-invariant operators, denoted $L_\alpha$ in his normalization, and the
associated homogeneous polynomial nullspaces
\[
  H_m^{(\alpha)}=\ker\bigl(L_\alpha|_{P_m(\Hh)}\bigr).
\]
His vector-field, sign, and scalar conventions differ from those used in the present paper.
After matching them, the case $\alpha=0$ corresponds to
$H_m(\Hh)$. Greiner's construction gives an explicit basis for this space
and proves
\[
  \dim H_m(\Hh)=m+1.
\]
Greiner and Koornwinder proved completeness of the
traces by solving the Dirichlet problem on the Kor\'anyi ball in~\cite{GreinerKoornwinder1983}; see also Dunkl's biorthogonal analysis~\cite{Dunkl1986}.
Greiner already noted that the system need not be
orthogonal. For the one-variable functions underlying the spherical basis,
Gasper proved a related full-circle complex-weight orthogonality and explained
why it does not persist on $(0,\pi)$ with a positive measure when the
parameters differ; refer to~\cite[pp.~1261--1262]{Gasper1981} for more details. In~\ref{sec:Hn-harmonic-basis}, we give a self-contained
triangular construction of all elements of $H_m(\Hh^n)$ and prove the
dimension formula without special functions.

\smallskip
Let $\rho$ be the Kor\'anyi-Folland homogeneous norm and $S_\rho^{(n)}$ the corresponding unit sphere:
\[
  \rho(z,t)=\bigl(|z|^4+t^2\bigr)^{1/4},
  \qquad
  S_\rho^{(n)}=\{\rho=1\}.
\]
Let $\sigma_n$ be the dilation-polar measure normalized in~\cref{sec:prelim}, set
\[
  H_m(S_\rho^{(n)})
  = \bigl\{h|_{S_\rho^{(n)}} \ : \ h\in H_m(\Hh^n) \bigr\},
\]
and define
\[
  \mathcal E_{-1}^{(n)}=\{0\},
  \qquad
  \mathcal E_m^{(n)}
  =\bigoplus_{j=0}^m{}^{\mathrm{alg}}H_j(S_\rho^{(n)}),
  \qquad
  \mathcal K_m^{(n)}
  =\mathcal E_m^{(n)}
     \cap\bigl(\mathcal E_{m-1}^{(n)}\bigr)^\perp,
\]
where orthogonality is taken in $L^2(S_\rho^{(n)},\sigma_n)$. Our first main result is a Heisenerg analogue of the classical spherical harmonic decomposition.

\begin{theorem}
\label{thm:L2-decomposition-H-intro}
For every $n\geq1$ and $m\geq0$, restriction to $S_\rho^{(n)}$ is injective
on $H_m(\Hh^n)$, the family of trace spaces
$\{H_m(S_\rho^{(n)})\}_{m\geq0}$ is algebraically direct, and its algebraic
span is uniformly dense in $C(S_\rho^{(n)})$.

The homogeneous degrees are not pairwise orthogonal. More precisely, for
every finite positive Borel measure of full support on $S_\rho^{(n)}$,
\[
  H_1(S_\rho^{(n)})\not\perp H_3(S_\rho^{(n)}).
\]
Nevertheless,
\[
  \mathcal E_m^{(n)}
  =\mathcal E_{m-1}^{(n)}
     \mathbin{\widehat\oplus}\mathcal K_m^{(n)},
  \qquad
  \dim\mathcal K_m^{(n)}
  =\binom{m+2n-1}{2n-1},
\]
and the Hilbert decomposition is
\begin{equation}\label{eq:intro-corrected-Hn}
  L^2(S_\rho^{(n)},\sigma_n)
  =\widehat{\bigoplus_{m=0}^{\infty}}\mathcal K_m^{(n)}.
\end{equation}
For $n=1$, this gives $\dim\mathcal K_m^{(1)}=m+1$.
\end{theorem}

Thus the statement is obtained by replacing the
homogeneous spaces by the orthogonal increments of their cumulative
filtration. Each element of $\mathcal K_m^{(n)}$ is still the trace of a
harmonic polynomial, but it generally mixes degree $m$ with lower
homogeneous degrees. The explicit obstruction and the complete proofs are
given in~\cref{sec:spherical,sec:higher-heisenberg}. The second main result is algebraic.

\begin{theorem}
\label{thm:eta-plus-decomposition-intro}
Let
\[
  w(z,t):=\eta_+^2(z,t)=|z|^2+4t.
\]
Then, for every $m\geq0$,
\begin{equation}\label{eq:intro-fischer}
  P_m(\Hh)
  =H_m(\Hh)\oplus wP_{m-2}(\Hh).
\end{equation}
Equivalently, every $p_m\in P_m(\Hh)$ has a unique representation
\[
  p_m=h_m+wq_{m-2},
  \qquad
  h_m\in H_m(\Hh),\quad q_{m-2}\in P_{m-2}(\Hh).
\]
The sum is algebraic. After restriction to $S_\rho$, its two summands are
not, in general, orthogonal in $L^2(S_\rho,\sigma)$, where
$\sigma=\sigma_1$.
\end{theorem}

The essential point is that the \emph{Fischer operator}
\[
  A_j:P_j(\Hh)\longrightarrow P_j(\Hh),
  \qquad
  A_jq=\Delta_H(wq),
\]
is an automorphism.~\cref{sec:eta-decomposition} develops the
decomposition and its Dirichlet interpretation, while
\ref{sec:fischer-appendix} contains the proof of
invertibility.

Finally, we state the general Carnot-group conclusion. Let $\mathbb G$ be a Carnot group with a
fixed sub-Laplacian $\mathcal L$, let $N$ be a continuous proper homogeneous
gauge, and let $\sigma_N$ be its dilation-polar measure on the corresponding unit sphere $S_N=\{N=1\}$.
Denoting by $P_m(\mathbb G)$ the weighted homogeneous coordinate
polynomials, set
\[
  H_m(\mathbb G)
  =\ker\bigl(\mathcal L|_{P_m(\mathbb G)}\bigr),
  \qquad
  H_m(S_N)
  =\bigl\{h|_{S_N} \ : \ h\in H_m(\mathbb G) \bigr\}.
\]
For these harmonic trace spaces, set
\[
  \mathcal E_m^N
  =\bigoplus_{j=0}^m{}^{\mathrm{alg}}H_j(S_N),
  \qquad
  \mathcal K_m^N
  =\mathcal E_m^N\cap(\mathcal E_{m-1}^N)^\perp,
  \qquad
  \mathscr H_N
  =\overline{\bigcup_{m\geq0}\mathcal E_m^N}^{\,L^2}.
\]
Here $\mathcal E_{-1}^N=\{0\}$. For all weighted polynomial traces, set
\[
  \mathscr A_m=\bigoplus_{j=0}^mP_j(\mathbb G),
  \qquad
  \mathscr V_m^N
  =\bigl\{p|_{S_N} \ : \ p\in\mathscr A_m \bigr\},
  \qquad
  \mathscr W_m^N
  =\mathscr V_m^N\cap(\mathscr V_{m-1}^N)^\perp,
\]
with the usual convention that $\mathscr V_{-1}^N=\{0\}$.
In the Carnot setting, we require an assumption. We denote by $(\mathrm{HD})_N$ the \emph{harmonic-density condition}
\[
  (\mathrm{HD})_N:
  \qquad
  \mathscr H_N=L^2(S_N,\sigma_N).
\]

\begin{theorem}
\label{thm:Carnot-L2-intro}
With the notation above,
\[
  \mathscr H_N
  =\widehat{\bigoplus_{m=0}^{\infty}}\mathcal K_m^N,
  \qquad
  \dim\mathcal K_m^N=\dim H_m(\mathbb G).
\]
This is a decomposition of $L^2(S_N,\sigma_N)$ iff the
separate harmonic-density condition $(\mathrm{HD})_N$ holds. 

By contrast,
the all-polynomial filtration is complete without any additional assumption. In other words, it holds
\[
  L^2(S_N,\sigma_N)
  =\widehat{\bigoplus_{m=0}^{\infty}}\mathscr W_m^N.
\]
\end{theorem}

The proof of this common principle appears once for the general case, in~\cref{sec:carnot-filtrations}. The Heisenberg conclusions then use
the Greiner--Koornwinder completeness result on top of it.

\subsection{A Euclidean-style decomposition on \texorpdfstring{$\Hh$}{H}}

As mentioned above, in the Euclidean setting every homogeneous polynomial $p_m$ 
of degree $m$ admits a unique decomposition
\[
  p_m(x)=h_m(x)+|x|^2q_{m-2}(x),
\]
where $h_m$ is harmonic and homogeneous of degree $m$. Equivalently,
\[
  P_m(\mathbb R^d)
  =\mathcal H_m(\mathbb R^d)\oplus|x|^2P_{m-2}(\mathbb R^d),
\]
where $\mathcal H_m$ denotes the Euclidean harmonic space.

This suggests looking for a degree-two multiplier on $\Hh$. A natural
first expression is
\[
  \eta(z,t)=\bigl(|z|^2+4|t|\bigr)^{1/2},
\]
but the absolute value prevents a polynomial factorization and is non-smooth
across $\{t=0\}$. We therefore use the homogeneous polynomial
\[
  w(z,t)=\eta_+^2(z,t)=|z|^2+4t.
\]
Only $w$ is used globally; the notation $\eta_+=\sqrt w$ makes sense as a
real function only on the patch $\{w>0\}$.

\cref{thm:eta-plus-decomposition-intro} is the precise Heisenberg
analogue of the Euclidean Fischer decomposition. Its proof is purely
algebraic and does not use the non-polynomial quantity
\[
  \rho^2(z,t)=\bigl(|z|^4+t^2\bigr)^{1/2}.
\]
Moreover, the Fischer automorphisms give a polynomial solution of a
Dirichlet problem on the unbounded paraboloid $\{w=1\}$.
The underlying degree-two correction is analogous to the construction of Kogoj
and Lanconelli for the heat operator under the same anisotropic scaling
\cite[Thm.~1.1]{KogojLanconelli2022}, although the operators and geometries
are different.

It is important not to confuse the two main themes: the Fischer sum
\eqref{eq:intro-fischer} is algebraic rather than orthogonal, while the
spherical Hilbert sum \eqref{eq:intro-corrected-Hn} is produced by
orthogonalizing a cumulative trace filtration.

\subsection{Contributions and organization of the paper}

The objective of this section is to clarify the contributions of this paper in comparison to the existing literature:

\begin{itemize}[leftmargin=*]
\item \emph{A polynomial realization of Heisenberg spherical
analysis.}
The homogeneous harmonic traces on $S_\rho^{(n)}$ are algebraically
direct and complete. Here, we give a measure-independent degree-one/degree-three
obstruction to their pairwise orthogonality and replace the homogeneous degree
Hilbert sum by the orthogonal increments $\mathcal K_m^{(n)}$. 

\item \emph{A Heisenberg analogue of the Euclidean Fischer decomposition.}
We establish the factorization
\[
  P_m(\Hh)=H_m(\Hh)\oplus\eta_+^2P_{m-2}(\Hh)
\]
as a unique algebraic decomposition. The proof relies primarily on the automorphisms
$q\mapsto\Delta_H(\eta_+^2q)$, and the same mechanism produces polynomial as a byproduct
Dirichlet corrections on $\{\eta_+^2=1\}$.

\item \emph{A precise general Carnot-group formulation.}
The harmonic filtration always decomposes its closed trace span and fills
the whole spherical $L^2$ space exactly when $(\mathrm{HD})_N$ holds. The filtration by all weighted polynomial traces is, on the other hand, always complete by Stone--Weierstrass.
\end{itemize}

The paper is organized as follows.
\cref{sec:prelim} gives the full
Heisenberg conventions, the sub-Laplacian, and the Kor\'anyi polar measure.
\cref{sec:carnot-filtrations} proves the common filtration argument.
\cref{sec:spherical,sec:higher-heisenberg} give the first
and higher-dimensional Heisenberg specializations. The algebraic
$\eta_+^2$ decomposition and its Dirichlet interpretation are developed in~\cref{sec:eta-decomposition}.  

Finally,~\ref{sec:Hn-harmonic-basis}
contains the constructive harmonic basis, while \ref{sec:fischer-appendix} proves the Fischer automorphism.

\section{Heisenberg groups, the sub-Laplacian, and polar measure}
\label{sec:prelim}

We now give the complete conventions summarized in~\cref{sec:heisenberg-intro}. The first Heisenberg group is
$\Hh=\Hh^1=\C\times\R$ with product
\[
  (z,t)(z',t')
  =\bigl(z+z',t+t'+2\operatorname{Im}(z\overline{z'})\bigr).
\]
More generally, on $\Hh^n=\C^n\times\R$ we use
\[
  (z,t)(z',t')
  =\bigl(z+z',t+t'
     +2\operatorname{Im}(z\cdot\overline{z'})\bigr),
  \qquad
  z\cdot\overline{z'}
  =\sum_{j=1}^n z_j\overline{z_j'}.
\]
Writing $z=x+\imath y$ with $x,y\in\R^n$, the product is equivalent to
\[
  (x,y,t)(x',y',t')
  =\bigl(x+x',y+y',
     t+t'+2(x'\cdot y-x\cdot y')\bigr),
\]
where $x \cdot y$ denotes the standard Euclidean product. Let $T=\partial_t$. The real left-invariant horizontal vector fields are
\[
  X_j=\partial_{x_j}+2y_j\partial_t,
  \qquad
  Y_j=\partial_{y_j}-2x_j\partial_t,
  \qquad 1\leq j\leq n.
\]
On the other hand, with the standard Wirtinger derivatives, the corresponding complex vector fields are
\[
  Z_j=\frac12(X_j-\imath Y_j)
     =\partial_{z_j}+\imath \overline z_j\,\partial_t,
  \qquad
  \overline Z_j=\frac12(X_j+\imath Y_j)
     =\partial_{\overline z_j}-\imath z_j\,\partial_t.
\]
The only nonzero commutators among these generators are
\[
  [X_j,Y_\ell]=-4\delta_{j\ell}T,
  \qquad
  [Z_j,\overline Z_\ell]
  =-2\imath \, \delta_{j\ell}T.
\]
Thus, there is a step-two stratification
\[
  \mathfrak h_n=V_1\oplus V_2,
  \qquad
  V_1=\operatorname{span}\bigl\{X_j,Y_j \ : \ 1\leq j\leq n\bigr\},
  \qquad
  V_2=\operatorname{span}\{T\}.
\]
The dilations are
\[
  \delta_s(z,t)=(sz,s^2t),
\]
and the homogeneous dimension of $\Hh^n$ is $Q=2n+2$. In this paper, we use the analytic-sign sub-Laplacian
\[
  \Delta_H
  =\sum_{j=1}^n(X_j^2+Y_j^2)
  =2\sum_{j=1}^n
      (Z_j\overline Z_j+\overline Z_jZ_j).
\]
Writing
\[
  \Delta_z
  =\sum_{j=1}^n
      (\partial_{x_j}^2+\partial_{y_j}^2),
  \qquad
  \mathcal R
  =\sum_{j=1}^n
      (y_j\partial_{x_j}-x_j\partial_{y_j}),
\]
one obtains the decomposition:
\begin{equation}\label{eq:Hn-expanded-Laplacian}
  \Delta_H
  =\Delta_z+4\mathcal R \, \partial_t+4|z|^2\partial_t^2.
\end{equation}
For $n=1$, using exponential coordinates $z=re^{\imath \theta}$ yields
\begin{equation}\label{eq:Delta-polar-rt}
  \Delta_H
  =\partial_r^2+\frac1r\partial_r+\frac1{r^2}\partial_\theta^2
   -4\partial_\theta\partial_t+4r^2\partial_t^2.
\end{equation}
The fields satisfy H\"ormander's bracket-generating condition. Hence, $\Delta_H$ is hypoelliptic, although it is not elliptic.
For the present
normalization, the positive operator $-\Delta_H$ has a fundamental solution
of the form
\[
  \Phi_n(z,t)
  =c_n\bigl(|z|^4+t^2\bigr)^{-n/2}
  \qquad c_n>0,
\]
as shown in~\cite{Folland1973,Folland1975}.
The value of $c_n$ depends on
the normalization of the operator and will not be needed below. The
fundamental solution is homogeneous of degree $2-Q = -2n$.

The weights in~\cref{def:homogeneous} are therefore the natural
ones: the horizontal variables have degree one, $t$ has degree two, and
\[
  \Delta_H P_m(\Hh^n)\subseteq P_{m-2}(\Hh^n).
\]
The Kor\'anyi--Folland gauge and its unit sphere are
\[
  \rho(z,t)=\bigl(|z|^4+t^2\bigr)^{1/4},
  \qquad
  S_\rho^{(n)}=\bigl\{(z,t) \ : \ \rho(z,t)=1\bigr\}.
\]
For each $n\geq1$, the dilation-polar measure $\sigma_n$ on
$S_\rho^{(n)}$ is normalized by
\begin{equation}\label{eq:Hn-polar-formula}
  \int_{\Hh^n}f(z,t)\,\dd z \dd t
  =\int_0^\infty\int_{S_\rho^{(n)}}
       f(\delta_s\omega)s^{2n+1}
       \dd\sigma_n(\omega)\dd s.
\end{equation}
For $n=1$ we abbreviate $S_\rho^{(1)}$ to $S_\rho$ and write
$\sigma=\sigma_1$. Thus~\eqref{eq:Hn-polar-formula} becomes
\begin{equation}\label{eq:polar-formula}
  \int_{\Hh}f(z,t)\,\dd z \dd t
  =\int_0^\infty\int_{S_\rho}
       f(\delta_s\omega)s^3
       \dd\sigma(\omega)\dd s,
\end{equation}
compare it with~\cite[Chapter~1]{FollandStein1982}. The parametrization
\begin{equation}\label{eq:sphere-param}
  (s,\phi,\theta)\mapsto
  \bigl(s\sqrt{\sin\phi}\,e^{\imath \theta},
        s^2\cos\phi\bigr),
  \quad
  s>0,\quad 0<\phi<\pi,\quad 0\leq\theta<2\pi,
\end{equation}
has Euclidean Jacobian $s^3$. Hence, with the normalization above,
\begin{equation}\label{eq:sigma-explicit}
  \dd\sigma=\dd\phi \dd\theta.
\end{equation}
The two characteristic points omitted by~\eqref{eq:sphere-param} have
$\sigma$-measure zero.

\begin{remark}
Rotations $R_\alpha(z,t)=(e^{\imath \alpha}z,t)$ preserve
$\Delta_H$, $\rho$, and $\sigma$, providing the genuine elementary
orthogonal splitting of the first spherical $L^2$ space.
\end{remark}

The corresponding
explicit higher-dimensional density and its $U(n)$ invariance are discussed in details
in~\cref{sec:higher-heisenberg}.

\section{Polar measure and polynomial filtrations on Carnot groups}
\label{sec:carnot-filtrations}

In this section, we prove the general filtration principle that will underlie the
Heisenberg specializations. Here, there are two different conclusions:
\begin{itemize}
    \item Harmonic polynomial traces always admit an orthogonal decomposition of their \emph{closed span}. Identifying it with the whole spherical $L^2$ space, however, requires a separate completeness theorem.
    \item The filtration by all weighted polynomial traces is always dense, by the Stone--Weierstrass theorem.
\end{itemize}

Let $\mathbb G$ be a connected, simply connected Carnot group with stratified
Lie algebra $\mathfrak g=V_1\oplus\cdots\oplus V_s$ satisfying
\[
  [V_1,V_j]=V_{j+1}\quad(1\leq j<s),
  \qquad [V_1,V_s]=\{0\}.
\]
Denote the dimension of each layer by $d_j=\dim V_j$ so that
\[
  Q=\sum_{j=1}^s j d_j
\]
is the homogeneous dimension. Choose exponential coordinates
$x=(x_{j,k})$, with $1\leq j\leq s$, $1\leq k\leq d_j$, adapted to the
stratification. Then
\[
  \delta_r(x_{j,k})=(r^j x_{j,k}),
  \qquad r>0,
\]
and the weighted degree of $x_{j,k}$ is $j$. With complex coefficients, let
\[
  P_m(\mathbb G)
  =\bigl\{p \ : \ p\circ\delta_r=r^m p\bigr\},
  \qquad
  \mathscr A_m=\bigoplus_{j=0}^m P_j(\mathbb G),
\]
with $\mathscr A_m=\{0\}$ for $m<0$. Let $X_1,\ldots,X_{d_1}$ be the
left-invariant vector fields corresponding to a basis of the first (aka horizontal) layer $V_1$, and define
\[
  \mathcal L :=\sum_{a=1}^{d_1}X_a^2,
  \qquad
  H_m(\mathbb G)=\ker\bigl(\mathcal L|_{P_m(\mathbb G)}\bigr).
\]
In adapted exponential coordinates the fields $X_a$ have polynomial
coefficients and homogeneous degree $-1$. In particular,
\[
  \mathcal L P_m(\mathbb G)\subseteq P_{m-2}(\mathbb G).
\]
Let $N$ be a continuous proper homogeneous gauge, namely a function satisfying $N(e)=0$,
$N(x)>0$ for $x\neq e$, and $N(\delta_r x)=rN(x)$. Write
\[
  S_N=\bigl\{x \ : \ N(x)=1 \bigr\},
  \qquad B_N=\bigl\{x \ : \ N(x)<1 \bigr\}.
\]
Then $S_N$ is compact and the map
\[
  (0,\infty)\times S_N\longrightarrow\mathbb G\setminus\{e\},
  \qquad (r,\omega)\longmapsto\delta_r\omega,
\]
is a homeomorphism.
Moreover, the unit ball $B_N$ is bounded and dilation-star-shaped with respect to $e$; in particular, it is connected.

\begin{proposition}[Polar measure and its support]
\label{prop:carnot-polar-support}
There is a finite Radon measure $\sigma_N$ on $S_N$ such that
\begin{equation}\label{eq:carnot-polar}
  \int_{\mathbb G}f(x)\dd x
  =\int_0^\infty\int_{S_N}f(\delta_r\omega)
     \dd\sigma_N(\omega)r^{Q-1}\dd r
\end{equation}
for every $f\in C_c(\mathbb G\setminus\{e\})$. Moreover,
$\operatorname{supp}\sigma_N=S_N$.
\end{proposition}

\begin{proof}
The polar formula~\eqref{eq:carnot-polar} follows by transporting Haar measure through the displayed
homeomorphism and using
\[
|\delta_r E|=r^Q|E|.
\]
This is the usual homogeneous-group polar
integration formula; refer to~\cite[Chapter~1]{FollandStein1982}. A monotone-class
argument extends the identity from compactly supported continuous functions
to nonnegative Borel functions.  

We verify the support statement, which will
be important below. If $U\subseteq S_N$ is nonempty
and relatively open and $0<a<b$, then
\[
  \mathcal C_{a,b}(U)
  := \bigl\{\delta_r\omega \ : \ a<r<b,\ \omega\in U \bigr\}
\]
is a nonempty open subset of $\mathbb G$, hence has positive Haar measure.
Applying~\eqref{eq:carnot-polar} to the indicator function of $\mathcal C_{a,b}(U)$ gives
\[
  |\mathcal C_{a,b}(U)|
  =\sigma_N(U)\int_a^b r^{Q-1}\dd r.
\]
Thus $\sigma_N(U)>0$. Since this holds for every nonempty relatively open
$U$, the measure has full support.
\end{proof}

We regard every trace space below as a subspace of
$L^2(S_N,\sigma_N)$. Notice that full support implies that a continuous function on
$S_N$ which represents the zero $L^2$ class vanishes everywhere.

\begin{proposition}[Injectivity and algebraic directness]
\label{prop:carnot-harmonic-direct}
For every $m\geq0$, restriction to $S_N$ is injective on
$P_m(\mathbb G)$, and hence on $H_m(\mathbb G)$. In addition, the trace
spaces
\[
  H_m(S_N)= \bigl\{h|_{S_N} \ : \ h\in H_m(\mathbb G) \bigr\}
\]
form an algebraically direct family. In other words,
\[
  \sum_{m=0}^M h_m|_{S_N}=0,
  \quad h_m\in H_m(\mathbb G)
  \implies
  h_0=\cdots=h_M=0.
\]
\end{proposition}

\begin{proof}
If $p\in P_m(\mathbb G)$ vanishes on $S_N$, then, for $x\neq e$, homogeneity gives
\[
  p(x)=N(x)^m \,
  p\bigl(\delta_{N(x)^{-1}}x\bigr)=0.
\]
Thus $p$ vanishes identically, proving injectivity.

For directness, set $u=\sum_{m=0}^M h_m$.  The operator $\mathcal L$ is
linear, so $\mathcal Lu=0$ in the bounded connected set $B_N$, and the
assumed trace identity gives 
\[
u=0 \qquad \text{on }\partial B_N=S_N.
\]
The vector fields
$X_1,\ldots,X_{d_1}$ satisfy H\"ormander's bracket-generating condition by
stratification. Bony's strong maximum principle~\cite{Bony1969}, applied to
$\pm\operatorname{Re}u$ and $\pm\operatorname{Im}u$, gives $u=0$ in $B_N$.
Hence $u$ vanishes identically.
Uniqueness of the weighted homogeneous expansion gives $h_m=0$ for
each $m$.
\end{proof}

The proposition allows an orthogonalization which does not assume a
general harmonic-completeness theorem. Specifically, set
\[
  \mathcal E_{-1}=\{0\},
  \qquad
  \mathcal E_m=\bigoplus_{j=0}^m{}^{\mathrm{alg}}H_j(S_N),
  \qquad
  \mathcal K_m=\mathcal E_m\cap\mathcal E_{m-1}^{\perp},
\]
and define the closed harmonic-trace span
\[
  \mathscr H_N
  =\overline{\bigcup_{m\geq0}\mathcal E_m}^{\,L^2(S_N,\sigma_N)}.
\]

\begin{theorem}[Harmonic filtration]
\label{thm:carnot-harmonic-filtration}
For every $m\geq0$,
\[
  \mathcal E_m
  =\mathcal E_{m-1}\mathbin{\widehat\oplus}\mathcal K_m,
  \qquad
  \dim\mathcal K_m=\dim H_m(\mathbb G),
\]
and
\begin{equation}\label{eq:carnot-harmonic-closed-span}
  \mathscr H_N
  =\widehat{\bigoplus_{m=0}^{\infty}}\mathcal K_m.
\end{equation}
If $\Pi_{m-1}$ is the orthogonal projection onto
$\mathcal E_{m-1}$, then the restriction
\[
  (I-\Pi_{m-1})|_{H_m(S_N)} :
  H_m(S_N)\longrightarrow\mathcal K_m
\]
is an isomorphism. Moreover, the sum
\eqref{eq:carnot-harmonic-closed-span} equals all of
$L^2(S_N,\sigma_N)$ if and only if the condition
$(\mathrm{HD})_N$ holds:
\begin{equation}\label{eq:carnot-HD}
  \overline{\bigoplus_{m\geq0}^{\mathrm{alg}}H_m(S_N)}^{\,L^2}
  =L^2(S_N,\sigma_N).
\end{equation}
\end{theorem}

\begin{proof}
Each $\mathcal E_m$ is finite-dimensional and
$\mathcal E_{m-1}\subseteq\mathcal E_m$, so its orthogonal complement in
$\mathcal E_m$ is $\mathcal K_m$. Algebraic directness from
\cref{prop:carnot-harmonic-direct} gives
\[
  \dim\mathcal E_m-\dim\mathcal E_{m-1}
  =\dim H_m(\mathbb G),
\]
proving the dimension formula. It also shows that the displayed
projection map is injective and hence an isomorphism. Taking the closure of
the increasing union of the $\mathcal E_m$ proves
\eqref{eq:carnot-harmonic-closed-span}. Finally, the last statement is exactly the
definition of the density condition~\eqref{eq:carnot-HD}.
\end{proof}

Dropping the harmonicity requirement gives an unconditional complete orthogonal decomposition. More precisely, set
\[
  \mathscr V_m
  =\bigl\{p|_{S_N} \ : \ p\in\mathscr A_m \bigr\},
  \qquad
  \mathscr W_m=\mathscr V_m\cap\mathscr V_{m-1}^{\perp},
  \qquad
  \mathscr V_{-1}=\{0\}.
\]

\begin{theorem}[All-polynomial filtration]
\label{thm:carnot-all-polynomial}
For every continuous proper homogeneous gauge $N$,
\begin{equation}\label{eq:carnot-all-polynomial-sum}
  L^2(S_N,\sigma_N)
  =\widehat{\bigoplus_{m=0}^{\infty}}\mathscr W_m.
\end{equation}
More precisely,
\[
  \mathscr V_m
  =\mathscr V_{m-1}\mathbin{\widehat\oplus}\mathscr W_m,
  \qquad
  \dim\mathscr W_m=\dim\mathscr V_m-\dim\mathscr V_{m-1}.
\]
If $I(S_N)$ is the ideal of coordinate polynomials vanishing on $S_N$, then
\begin{equation}\label{eq:carnot-Hilbert-function}
  \dim\mathscr V_m
  =\dim\mathscr A_m-\dim\bigl(I(S_N)\cap\mathscr A_m\bigr).
\end{equation}
\end{theorem}

\begin{proof}
The restrictions of coordinate polynomials form a unital self-conjugate
algebra which separates points of the compact set $S_N$. Thus, applying the
Stone--Weierstrass theorem gives
\[
  C(S_N)
  =\overline{\bigcup_{m\geq0}\mathscr V_m}^{\,\|\cdot\|_\infty}.
\]
Since $\sigma_N$ is finite, the same union is dense in $L^2$ as well. Orthogonally
splitting each pair of consecutive finite-dimensional spaces proves
\eqref{eq:carnot-all-polynomial-sum}.
Finally, full support of $\sigma_N$
shows that the kernel of
$\mathscr A_m\to\mathscr V_m$ is exactly
$I(S_N)\cap\mathscr A_m$, which proves
\eqref{eq:carnot-Hilbert-function}.
\end{proof}

For the Kor\'anyi sphere in every Heisenberg group, the dimensions in the
last theorem are explicit. Write
\[
  \Hh^n=\mathbb R^{2n}\times\mathbb R,
  \qquad
  \delta_r(z,t)=(rz,r^2t),
  \qquad
  \rho(z,t)=\bigl(|z|^4+t^2\bigr)^{1/4}.
\]

\begin{proposition}
\label{prop:Hn-W-dimensions}
Let $S_\rho^{(n)}=\{\rho=1\}$ and set
\[
\begin{aligned}
  \mathscr A_m^{(n)}&=\bigoplus_{j=0}^mP_j(\Hh^n),\\
  \mathscr V_m^{(n)}
    &=\bigl\{p|_{S_\rho^{(n)}} \ : \ p\in\mathscr A_m^{(n)}\bigr\},\quad
  \mathscr W_m^{(n)}
    =\mathscr V_m^{(n)}\cap
      \bigl(\mathscr V_{m-1}^{(n)}\bigr)^\perp.
\end{aligned}
\]
Here $\mathscr A_m^{(n)}=\mathscr V_m^{(n)}=\{0\}$ for $m<0$. Then
\begin{equation}\label{eq:Hn-V-dimension}
  \dim\mathscr V_m^{(n)}
  =\binom{m+2n}{2n}+\binom{m+2n-2}{2n},
  \qquad m\geq0,
\end{equation}
and
\begin{equation}\label{eq:Hn-W-dimension}
  \dim\mathscr W_m^{(n)}
  =\binom{m+2n-1}{2n-1}
   +\binom{m+2n-3}{2n-1},
  \qquad m\geq0.
\end{equation}
Equivalently,
\[
  \sum_{m=0}^{\infty}(\dim\mathscr W_m^{(n)})q^m
  =\frac{1+q^2}{(1-q)^{2n}}.
\]
In particular, for $n=1$,
\[
  \dim\mathscr V_m^{(1)}=m^2+m+1,
  \qquad
  \dim\mathscr W_m^{(1)}= \begin{cases} 1 & \text{if } m = 0, \\ 2m & \text{if } m \ge 1. \end{cases}
\]
\end{proposition}

\begin{proof}
Define the function
\[
  F(z,t)=|z|^4+t^2-1.
\]
The vanishing ideal of $S_\rho^{(n)}$ is $(F)$. Indeed, division by the monic
quadratic $F$ in the variable $t$ gives
\[
  p=Fq+a(z)t+b(z).
\]
Here $a$ and $b$ are polynomials in the $2n$ real horizontal coordinates.
For $|z|<1$, evaluation at
$t=\pm\sqrt{1-|z|^4}$ and subtraction show that $a$ vanishes on an open
ball; addition then gives the same conclusion for $b$. Thus $a=b=0$.

If $q_d$ is the top weighted homogeneous part of a nonzero polynomial $q$,
then the top part of $Fq$ is $(|z|^4+t^2)q_d$ and has degree $d+4$.
Hence,
\[
  \ker(\mathscr A_m^{(n)}\longrightarrow\mathscr V_m^{(n)})
  =F\mathscr A_{m-4}^{(n)},
\]
and therefore
\[
  \dim\mathscr W_m^{(n)}
  =\dim P_m(\Hh^n)-\dim P_{m-4}(\Hh^n).
\]
There are $2n$ variables of weight one and one variable of weight two, so
\[
  \dim P_m(\Hh^n)
  =\sum_{k=0}^{\lfloor m/2\rfloor}
    \binom{m-2k+2n-1}{2n-1}.
\]
Subtracting the same expression with $m$ replaced by $m-4$ cancels all but
the first two terms and gives~\eqref{eq:Hn-W-dimension}. Its generating
function is $(1+q^2)/(1-q)^{2n}$. Since the $\mathscr V_m^{(n)}$ are the
cumulative spaces, their dimension generating function is
$(1+q^2)/(1-q)^{2n+1}$, whose coefficient of $q^m$ is the right-hand side
of~\eqref{eq:Hn-V-dimension}. 
\end{proof}

\begin{remark}
\label{rem:carnot-polar-limit}
Formula~\eqref{eq:carnot-polar} is a statement about Haar measure; it does
not provide a polar-coordinate formula for the differential operator
$\mathcal L$. Homogeneity alone says that, if
$u(\delta_r\omega)=r^m f(\omega)$, then
\[
  (\mathcal Lu)(\delta_r\omega)=r^{m-2}(T_m f)(\omega)
\]
for some degree-dependent operator $T_m$. It does not imply that
\[
T_m=\mathcal L_{S_N}+\lambda_m
\]
for a symmetric angular operator
$\mathcal L_{S_N}$. In concrete non-commutative examples
radial and tangential derivatives can be coupled, so the Euclidean
eigenvalue argument for orthogonality of different degrees is unavailable.

Nor does~\eqref{eq:carnot-polar} imply the density condition
$(\mathrm{HD})_N$ in~\eqref{eq:carnot-HD}. Harmonic polynomials do not form an algebra, so the
Stone--Weierstrass argument used for all polynomials cannot be applied. Harmonic density requires additional analytic input, such as a
polynomially compatible Dirichlet or completeness theorem for the chosen
group and gauge. Moreover, a homogeneous gauge need not have polynomial
$N^2$, so a Euclidean Fischer expansion in powers of $N^2$ may not be available
in general. The universally valid conclusions are therefore the harmonic
decomposition~\eqref{eq:carnot-harmonic-closed-span} on $\mathscr H_N$ and
the complete all-polynomial decomposition
\eqref{eq:carnot-all-polynomial-sum}.
\end{remark}

\section{Spherical decomposition on the first Heisenberg group}
\label{sec:spherical}

In this section, we specialize the results obtained in~\cref{sec:carnot-filtrations} to the first Heisenberg group $\Hh$.

\subsection{Homogeneous harmonic traces and non-orthogonality}

For $m\geq0$ set
\[
  H_m(S_\rho)= \bigl\{h|_{S_\rho} \ : \ h\in H_m(\Hh) \bigr\}.
\]
We use the complex $L^2$ inner product
\[
  \langle f,g\rangle_\sigma
  =\int_{S_\rho}f\overline g\,\dd\sigma.
\]

\begin{proposition}[Failure of orthogonality]\label{prop:counterexample}
The spaces $H_m(S_\rho)$ are not pairwise orthogonal. For instance,
\[
  1\in H_0(S_\rho),
  \qquad
  h_4=t^2-\frac12(x^2+y^2)^2\in H_4(\Hh),
\]
but they are not orthogonal:
\[
  \langle h_4,1\rangle_\sigma=\frac{\pi^2}{2}\neq0.
\]
\end{proposition}

\begin{proof}
Since
\[
  \Delta_H(t^2)=8(x^2+y^2),
  \qquad
  \Delta_H\bigl((x^2+y^2)^2\bigr)=16(x^2+y^2),
\]
the polynomial $h_4$ is harmonic. On~\eqref{eq:sphere-param} with $s=1$ it has the value
\[
  h_4=\cos^2\phi-\frac12\sin^2\phi.
\]
Using~\eqref{eq:sigma-explicit}, the scalar product $\langle h_4,1\rangle_\sigma$ equals
\[
  \int_{S_\rho}h_4\,\dd\sigma
  =2\pi\int_0^\pi
       \left(\cos^2\phi-\frac12\sin^2\phi\right)\dd\phi
  =\frac{\pi^2}{2}.
\]
\end{proof}

\begin{remark}\label{rem:operator-pencil}
A direct calculation in the coordinates~\eqref{eq:sphere-param} gives, with
$a=\sin\phi$ and $b=\cos\phi$, the following formula for the Laplacian operator:
\[
\Delta_H = a\left(\partial_s^2+\frac3s\partial_s\right)
 -\frac{2b}{s}\partial_{s\theta} +\frac1{s^2}\left(
4a\partial_{\phi\phi}+4b\partial_\phi
+4a\partial_{\phi\theta}+a^{-1}\partial_{\theta\theta}
\right).
\]
Thus the radial and angular derivatives are coupled. Hence, if $u=s^mf(\phi,\theta)$ is homogeneous and harmonic, then $f$ satisfies an $m$-dependent operator pencil, not an eigenvalue equation for one fixed symmetric spherical operator.
\end{remark}

\subsection{Angular orthogonality and completeness}

The rotation action does yield an orthogonal decomposition and it can be proved rather easily. For
$\nu\in\mathbb Z$, define
\[
  \mathscr L_\nu
  := \bigl\{f\in L^2(S_\rho,\sigma) \ : \ f(R_\alpha\omega)
       =e^{\imath \nu\alpha}f(\omega)\text{ for all }\alpha \bigr\}.
\]

\begin{proposition}[Angular orthogonality]\label{prop:angular}
There holds
\[
  L^2(S_\rho,\sigma)=\widehat{\bigoplus_{\nu\in\mathbb Z}}\mathscr L_\nu.
\]
Moreover,
\[
  H_m(S_\rho)
  =\bigoplus_{\substack{|\nu|\leq m\\m-|\nu|\text{ even}}}
       H_{m,\nu}(S_\rho),
  \qquad
  H_{m,\nu}(S_\rho) := H_m(S_\rho)\cap\mathscr L_\nu,
\]
and every nonzero $H_{m,\nu}(S_\rho)$ is one-dimensional. For each fixed
$\nu$,
\[
  \overline{\operatorname{span}
  \bigl\{H_{|\nu|+2j,\nu}(S_\rho) \ : \ j\geq0 \bigr\}}^{\,L^2}=\mathscr L_\nu.
\]
\end{proposition}

\begin{proof}
The first assertion is the ordinary Fourier decomposition in the variable
$\theta$, while the description and dimension of the harmonic rotation types
follow from Greiner's explicit basis~\cite[Thm.~8.5]{Greiner1980}.  

For fixed-mode completeness, take $(k,l)=(\nu,0)$ when $\nu\geq0$ and
$(k,l)=(0,-\nu)$ when $\nu<0$ in the notation of Greiner and Koornwinder,
and apply their completeness theorem
\cite[Thm.~5.2 and Cor.~5.3]{GreinerKoornwinder1983}.
\end{proof}

\begin{remark}
The map $J(z,t)=(-z,t)$ preserves $\sigma$. Every weighted homogeneous polynomial of degree $m$ satisfies 
\[
p\circ J=(-1)^mp.
\]
Consequently, $H_m(S_\rho)\perp H_k(S_\rho)$ whenever $m$ and $k$ have opposite parity. The counterexample in~\cref{prop:counterexample} shows that no corresponding assertion holds for arbitrary distinct degrees of the same parity.
\end{remark}

We next distinguish completeness from degree orthogonality.

\begin{theorem}[Completeness of Heisenberg harmonic traces]\label{thm:harmonic-completeness}
For every $m\geq0$, restriction from $H_m(\Hh)$ to $S_\rho$ is injective and
\[
  \dim H_m(S_\rho)=m+1.
\]
The finite sum of the spaces $H_m(S_\rho)$ is algebraically direct, and
\begin{equation}\label{eq:dense-harmonics}
  C(S_\rho)
  =\overline{\bigoplus_{m\geq0}^{\mathrm{alg}}H_m(S_\rho)}^{\,\|\cdot\|_\infty}.
\end{equation}
In particular, the same algebraic sum is dense in $L^p(S_\rho,\sigma)$ for $1\leq p<\infty$.
\end{theorem}

\begin{proof}
If a homogeneous polynomial vanishes on $S_\rho$, homogeneity shows that it vanishes on $\Hh\setminus\{0\}$ and hence identically. The dimension formula is due to Greiner~\cite{Greiner1980} (follows independently from~\cref{cor:dimensions}). Algebraic directness is the specialization of~\cref{prop:carnot-harmonic-direct} to $\mathbb G=\Hh$ and $N=\rho$.

\smallskip
For completeness, let $f\in C(S_\rho)$. Its Fej\'er means for the rotation action converge uniformly to $f$ and contain only finitely many angular modes. Greiner and Koornwinder's fixed-mode completeness theorem~\cite[Thm.~5.2 and Cor.~5.3]{GreinerKoornwinder1983}, based on solvability of the Heisenberg Dirichlet problem, uniformly approximates each of those modes by the corresponding harmonic traces; see also~\cite{Gaveau1977,Jerison1981,Jerison1981B}. A diagonal choice proves~\eqref{eq:dense-harmonics}. Finally, uniform density implies $L^p$ density because $S_\rho$ is compact and $\sigma$ is finite.
\end{proof}

\subsection{The Hilbert decomposition}

We can now state the orthogonal decomposition in the first Heisenberg group. For this, set
\[
  \mathcal E_{-1}=\{0\},
  \qquad
  \mathcal E_m=\bigoplus_{j=0}^m{}^{\mathrm{alg}}H_j(S_\rho),
  \qquad
  \mathcal K_m=\mathcal E_m\cap\mathcal E_{m-1}^{\perp}.
\]

\begin{theorem}\label{thm:corrected-L2}
For every $m\geq0$, it holds
\[
  \mathcal E_m=\mathcal E_{m-1}\mathbin{\widehat\oplus}\mathcal K_m,
  \qquad
  \dim\mathcal K_m=m+1,
\]
and
\begin{equation}\label{eq:corrected-L2}
  L^2(S_\rho,\sigma)
  =\widehat{\bigoplus_{m=0}^\infty}\mathcal K_m.
\end{equation}
If $\Pi_{m-1}$ denotes orthogonal projection onto $\mathcal E_{m-1}$, then the restriction
\[
  (I-\Pi_{m-1})|_{H_m(S_\rho)}:
  H_m(S_\rho)\longrightarrow\mathcal K_m
\]
is an isomorphism.
\end{theorem}

\begin{proof}
Apply~\cref{thm:carnot-harmonic-filtration} with
$\mathbb G=\Hh$ and $N=\rho$. The density assertion in
\cref{thm:harmonic-completeness} verifies $(\mathrm{HD})_\rho$, while
$\dim H_m(\Hh)=m+1$.
\end{proof}

\begin{remark}
Every element of $\mathcal K_m$ is the trace of a harmonic polynomial of weighted degree at most $m$. In general it is not homogeneous: orthogonal projection subtracts lower-degree harmonic components. Hence,~\cref{thm:corrected-L2} is a genuine Hilbert direct sum, but its summands must not be identified with the original spaces $H_m(S_\rho)$.
\end{remark}

\subsection{Comparison with the all-polynomial filtration}
\label{sec:all-polynomials}

For comparison with the harmonic filtration, we also state the
$n=1$ specialization of the all-polynomial construction from~\cref{sec:carnot-filtrations}. Let
\[
  \mathscr A_m=\bigoplus_{j=0}^m P_j(\Hh),
  \qquad \mathscr A_m=\{0\}\quad(m<0),
\]
and let
\[
  \mathscr V_m= \bigl\{p|_{S_\rho} \ : \ p\in\mathscr A_m\bigr\},
  \qquad
  \mathscr W_m=\mathscr V_m\cap\mathscr V_{m-1}^{\perp}.
\]

\begin{theorem}\label{thm:all-poly}
For every $m\geq0$,
\[
  \dim\mathscr V_m=m^2+m+1,
\]
while
\[
  \dim\mathscr W_0=1,
  \qquad
  \dim\mathscr W_m=2m\quad(m\geq1).
\]
Moreover,
\begin{equation}\label{eq:all-poly-L2}
  L^2(S_\rho,\sigma)
  =\widehat{\bigoplus_{m=0}^\infty}\mathscr W_m.
\end{equation}
\end{theorem}

\begin{proof}
Apply~\cref{thm:carnot-all-polynomial} with
$\mathbb G=\Hh$ and $N=\rho$. The dimension formulas are exactly the
$n=1$ specialization of~\cref{prop:Hn-W-dimensions}.
\end{proof}

\begin{remark}
Order the weighted monomials by degree and apply Gram--Schmidt, discarding zero residuals. The residuals first appearing at degree $m$ form an orthonormal basis of $\mathscr W_m$. These summands depend on $\sigma$ and are not harmonic. Their dimensions also display the distinction: for $m\geq2$,
\[
  \dim\mathscr W_m=2m,
  \qquad
  \dim H_m(S_\rho)=m+1.
\]
\end{remark}

\section{Higher-dimensional Heisenberg groups}
\label{sec:higher-heisenberg}

The preceding spherical discussion has a natural higher-dimensional form,
provided that homogeneous degree is not confused with orthogonality.

\subsection{The higher-dimensional Kor\'anyi sphere}

We use the group, vector-field, and homogeneity conventions of~\cref{sec:prelim} and~\cref{def:homogeneous}. It remains
to describe the explicit density of the higher-dimensional polar measure. Consider the unit sphere
\[
  S_\rho^{(n)}= \bigl\{(z,t)\in\Hh^n \ : \ \rho(z,t)=1 \bigr\},
\]
and let $\sigma_n$ be the measure defined by
\eqref{eq:Hn-polar-formula}.
If $\xi\in\mathbb S^{2n-1}\subset\mathbb C^n$, then
\begin{equation}\label{eq:Hn-polar-parametrization}
  (s,\phi,\xi)\longmapsto
  \bigl(s\sqrt{\sin\phi}\,\xi,s^2\cos\phi\bigr),
  \qquad 0<\phi<\pi,\quad s>0,
\end{equation}
has Euclidean Jacobian
$s^{2n+1}(\sin\phi)^{n-1}$. With our normalization, it yields
\begin{equation}\label{eq:Hn-polar-measure}
  \dd\sigma_n
  =(\sin\phi)^{n-1}\dd\phi\,
     \dd\omega_{2n-1}(\xi),
\end{equation}
where $\omega_{2n-1}$ is the usual surface measure on
$\mathbb S^{2n-1}$. In particular, $\sigma_n$ is finite, has full support,
and is invariant under the natural action of $U(n)$.

As in the introduction, we write
\[
  H_m(S_\rho^{(n)})
  = \bigl\{h|_{S_\rho^{(n)}} \ : \ h\in H_m(\Hh^n) \bigr\}.
\]

\subsection{A measure-independent obstruction to degree orthogonality}

The failure of degree orthogonality is not peculiar to $\Hh$, nor
can it be repaired by replacing the polar measure with another
surface measure.

\begin{proposition}[Universal obstruction]
\label{prop:Hn-universal-obstruction}
Let $\mu$ be any finite positive Borel measure on $S_\rho^{(n)}$ with full
support.
Then 
\[
H_1(S_\rho^{(n)}) \not\perp H_3(S_\rho^{(n)}) \qquad \text{in } L^2(S_\rho^{(n)},\mu).
\]
\end{proposition}

\begin{proof}
Consider the polynomials
\[
  h_1(z,t)=z_1,
  \qquad
  h_3(z,t)=z_1\bigl(|z|^2-\imath(n+1)t\bigr).
\]
Both have the indicated Heisenberg degrees. Moreover,
\[
  \Delta_z(z_1|z|^2)=4(n+1)z_1,
  \qquad
  \mathcal R z_1=-\imath z_1,
  \qquad
  \Delta_H(z_1t)=-4\imath z_1.
\]
Formula~\eqref{eq:Hn-expanded-Laplacian} therefore gives
$\Delta_Hh_1=\Delta_Hh_3=0$. On the other hand,
\[
  \operatorname{Re}\bigl(h_3\overline{h_1}\bigr)
  =|z_1|^2|z|^2.
\]
This is nonnegative everywhere and strictly positive on a nonempty open
subset of the unit sphere $S_\rho^{(n)}$. Since $\mu$ has full support,
\[
  \operatorname{Re}\langle h_3,h_1\rangle_{L^2(\mu)}
  =\int_{S_\rho^{(n)}}|z_1|^2|z|^2\dd\mu>0.
\]
Thus the two trace spaces cannot be orthogonal.
\end{proof}

\subsection{Polynomial surjectivity and dimensions}

Although multiplication by $\rho^2$ does not preserve polynomials, the
dimension of $H_m(\Hh^n)$ still has a simple algebraic formula.
Every weighted homogeneous polynomial has a unique splitting
\begin{equation}\label{eq:Hn-weighted-polynomial-splitting}
  P_m(\Hh^n)
  =\bigoplus_{\gamma=0}^{\lfloor m/2\rfloor}
       t^\gamma\mathscr P_{m-2\gamma}(\mathbb R^{2n}),
\end{equation}
where $\mathscr P_d(\mathbb R^{2n})$ denotes the ordinary homogeneous
horizontal polynomials of degree $d$. The triangular argument
in~\ref{sec:Hn-harmonic-basis} proves that
\[
  \Delta_H:P_m(\Hh^n)\longrightarrow P_{m-2}(\Hh^n)
\]
is surjective for every $m\geq2$ and explicitly parametrizes its kernel.

\begin{theorem}
\label{thm:Hn-harmonic-dimension}
For every $m\geq0$,
\begin{equation}\label{eq:Hn-harmonic-dimension}
  \dim H_m(\Hh^n)=\binom{m+2n-1}{2n-1}.
\end{equation}
In particular, $\dim H_m(\Hh)=m+1$.
\end{theorem}

\begin{proof}
\cref{cor:Hn-harmonic-dimension} gives
$\binom{m+2n-1}{m}$, which equals the displayed expression by binomial
symmetry.
\end{proof}

\subsection{Completeness and the Hilbert decomposition}

Restriction from $H_m(\Hh^n)$ to $S_\rho^{(n)}$ is injective: if a
homogeneous polynomial vanishes on the sphere, homogeneity makes it
vanish on $\Hh^n\setminus\{0\}$. However, we can prove an even more general property.

\begin{theorem}
\label{thm:Hn-trace-completeness}
The finite sum of the spaces $H_m(S_\rho^{(n)})$ is algebraically direct,
and
\begin{equation}\label{eq:Hn-uniform-completeness}
  C(S_\rho^{(n)})
  =\overline{\bigoplus_{m\geq0}^{\mathrm{alg}}
       H_m(S_\rho^{(n)})}^{\,\|\cdot\|_\infty}.
\end{equation}
Consequently, the same algebraic sum is dense in
$L^p(S_\rho^{(n)},\sigma_n)$ for $1\leq p<\infty$.
\end{theorem}

\begin{proof}
Once again, the algebraic directness is the specialization of
\cref{prop:carnot-harmonic-direct} to $\mathbb G=\Hh^n$ and $N=\rho$.

\smallskip
For uniform density, convolve a given $f\in C(S_\rho^{(n)})$ with an
approximate identity on the compact group $U(n)$ whose spectrum
is finite. The resulting functions converge uniformly to $f$ and contain
only finitely many $U(n)$ types. Greiner and Koornwinder's fixed-type
completeness theorem~\cite[Thm.~5.2 and Cor.~5.3]{GreinerKoornwinder1983}
uniformly approximates each of those components by harmonic traces; a
diagonal choice proves~\eqref{eq:Hn-uniform-completeness}. The required
Dirichlet solvability is available from~\cite{Gaveau1977,Jerison1981,Jerison1981B}.
Finally, uniform density implies
$L^p$ density because the sphere is compact and $\sigma_n$ is finite.
\end{proof}

We conclude the section, as in the case of the first Heisenberg group, by providing the orthogonal decomposition. For this, set
\[
  \mathcal E_{-1}^{(n)}=\{0\},
  \qquad
  \mathcal E_m^{(n)}
  =\bigoplus_{j=0}^m{}^{\mathrm{alg}}H_j(S_\rho^{(n)}),
  \qquad
  \mathcal K_m^{(n)}
  =\mathcal E_m^{(n)}\cap
       \bigl(\mathcal E_{m-1}^{(n)}\bigr)^\perp,
\]
where orthogonality is taken in $L^2(S_\rho^{(n)},\sigma_n)$.

\begin{theorem}
\label{thm:Hn-corrected-Hilbert-sum}
For every $m\geq0$,
\[
  \mathcal E_m^{(n)}
  =\mathcal E_{m-1}^{(n)}\mathbin{\widehat\oplus}
     \mathcal K_m^{(n)},
  \qquad
  \dim\mathcal K_m^{(n)}=\binom{m+2n-1}{2n-1},
\]
and
\begin{equation}\label{eq:Hn-corrected-Hilbert-sum}
  L^2(S_\rho^{(n)},\sigma_n)
  =\widehat{\bigoplus_{m=0}^\infty}\mathcal K_m^{(n)}.
\end{equation}
If $\Pi_{m-1}^{(n)}$ is orthogonal projection onto
$\mathcal E_{m-1}^{(n)}$, then the restriction
\[
  (I-\Pi_{m-1}^{(n)})|_{H_m(S_\rho^{(n)})}:
  H_m(S_\rho^{(n)})\longrightarrow\mathcal K_m^{(n)}
\]
is an isomorphism.
\end{theorem}

\begin{proof}
Apply~\cref{thm:carnot-harmonic-filtration} with
$\mathbb G=\Hh^n$ and $N=\rho$. The density statement in
\cref{thm:Hn-trace-completeness} verifies $(\mathrm{HD})_\rho$, and
\cref{thm:Hn-harmonic-dimension} supplies the dimension.
\end{proof}

\subsection{The \texorpdfstring{$U(n)$}{U(n)}-type refinement}

The preceding decomposition is compatible with the classical explicit
description of Heisenberg harmonics. Let $\mathcal H_{a,b}(\mathbb C^n)$
be the ordinary complex harmonic polynomials that are bihomogeneous of
degree $a$ in $z$ and degree $b$ in $\overline z$. Greiner and
Koornwinder show that, for $m=a+b+2q$, the corresponding copy of
$\mathcal H_{a,b}(\mathbb C^n)$ inside $H_m(\Hh^n)$ is represented by
\begin{equation}\label{eq:Hn-GK-types}
  C_q^{(n/2+b,n/2+a)}\bigl(t+\mathrm{i}|z|^2\bigr)Y(z),
  \qquad Y\in\mathcal H_{a,b}(\mathbb C^n),
\end{equation}
where $C_q^{(\alpha,\beta)}(\zeta)$ denotes the Greiner--Koornwinder
homogeneous polynomial in $\zeta$ and $\overline\zeta$ defined by their
generating function. The displayed formula is given by
\cite[Thm.~3.5]{GreinerKoornwinder1983}.
Moreover,~\cite[Prop.~3.4(a),(b)]{GreinerKoornwinder1983} shows that the
$U(n)$ types are irreducible and pairwise inequivalent. Since the action
preserves $\sigma_n$, distinct types $(a,b)$ are orthogonal. For a fixed
type $(a,b)$, however, the same irreducible representation occurs for every
$q=0,1,2,\ldots$, and these repeated radial copies need not be orthogonal.
\cref{prop:Hn-universal-obstruction} is already the first such
failure, for $(a,b)=(1,0)$ and $q=0,1$.

Because every $\mathcal E_m^{(n)}$ is $U(n)$-invariant, its orthogonal
projection commutes with $U(n)$. Thus the passage from $H_m$ to
$\mathcal K_m^{(n)}$ leaves the $U(n)$ types separate: equivalently, one may
apply Gram--Schmidt only along each repeated radial chain
in~\eqref{eq:Hn-GK-types}.
This gives a concrete $U(n)$-equivariant orthogonal system of harmonic
polynomial traces. It retains the leading homogeneous degree and all
angular representation data, while necessarily mixing lower homogeneous
degrees within a fixed radial chain.

\section{An \texorpdfstring{$\eta$}{eta}-based decomposition of polynomials on \texorpdfstring{$\Hh$}{H}}
\label{sec:eta-decomposition}

We now discuss the Euclidean-style decomposition~\eqref{eq:intro-fischer}. A first analogy with a homogeneous gauge suggests the
degree-two expression $|z|^2+4|t|$, but the absolute value prevents a
polynomial factorization and is non-smooth across $\{t=0\}$. We therefore use
the homogeneous polynomial multiplier
\[
  w(x,y,t)=\eta_+^2(x,y,t)=x^2+y^2+4t.
\]
The notation $\eta_+$ is useful on the region where $w>0$, but every argument
below is an identity between polynomials on all of $\Hh$. In particular, no
claim that $\eta_+$ is a global homogeneous norm is needed.
More precisely, one may set $\eta_+=\sqrt w$ on $\{w>0\}$, and throughout
the formulas below
\[
  \eta_+^{2k}:=(\eta_+^2)^k=w^k.
\]
Thus, only globally defined polynomial powers occur.

\subsection{The Fischer operator}

The polynomial $w$ is homogeneous of degree two. In addition, a direct calculation gives
\[
  Xw=2x+8y,
  \qquad
  Yw=2y-8x,
  \qquad
  \Delta_Hw=4.
\]
Thus $w=\eta_+^2$ is not $\Delta_H$-harmonic; its role is that of a
special polynomial multiplier.
For $j\geq0$ define the Fischer operator
\[
  A_j:P_j(\Hh)\longrightarrow P_j(\Hh),
  \qquad
  A_jq=\Delta_H(wq).
\]

\begin{lemma}\label{lem:Fischer-auto}
For every $j\geq0$, the map $A_j$ is an automorphism.
\end{lemma}

The proof is technical and it is thus postponed to~\ref{sec:fischer-appendix}.
It resolves the
polynomial spaces into rotation types:
\begin{itemize}
    \item the nonzero types are controlled by a definite imaginary part;
    \item while the invariant type reduces to a nonvanishing
Legendre value.
\end{itemize}

\subsection{A Dirichlet-type algebraic lemma}

The same automorphisms prove the stronger inhomogeneous correction used in the Dirichlet interpretation.

\begin{proposition}\label{prop:Dirichlet-correction}
Let $p\in P_m(\Hh)$. Then, there is a unique polynomial $q$ of weighted degree at most $m-2$ such that
\begin{equation}\label{eq:inhomogeneous-correction}
  \Delta_H\bigl((1-w)q\bigr)=-\Delta_Hp.
\end{equation}
Consequently,
\[
  u=p+(1-w)q
\]
is a harmonic polynomial and $u=p$ on the paraboloid
\[
  \{w=1\}=\left\{t=\frac{1-|z|^2}{4}\right\}.
\]
\end{proposition}

\begin{proof}
If $m=0$ or $m=1$, then $\Delta_Hp=0$ and the only polynomial of weighted degree at most $m-2$ is $q=0$. We can thus assume that $m\geq2$.

\smallskip
Write $q=\sum_jq_j$ in weighted homogeneous components and set $q_j=0$ outside $0\leq j\leq m-2$. The degree-$j$ component of the LHS of~\eqref{eq:inhomogeneous-correction} is
\begin{equation}\label{eq:degree-recursion}
  \Delta_Hq_{j+2}-A_jq_j.
\end{equation}
Choose $q_{m-2}=A_{m-2}^{-1}(\Delta_Hp)$, and, descending through degrees of the same parity, define
\[
  q_j=A_j^{-1}(\Delta_Hq_{j+2}).
\]
Set the components of the opposite parity equal to zero. Formula~\eqref{eq:degree-recursion} then proves~\eqref{eq:inhomogeneous-correction}. The same argument, with zero RHS, proves uniqueness.
\end{proof}

\begin{remark}
By decomposing an arbitrary polynomial datum into weighted homogeneous
components and summing the corresponding corrections, the proposition also
gives a unique polynomial $q$ for arbitrary polynomial $p$, with the same
identity~\eqref{eq:inhomogeneous-correction}.
\end{remark}

\subsection{Euclidean-style decomposition}

\begin{theorem}[Decomposition via $\eta_+^2$]\label{thm:eta-decomposition}
For every $m\geq0$, there is an algebraic direct-sum decomposition
\begin{equation}\label{eq:eta-direct}
  P_m(\Hh)=H_m(\Hh)\oplus \eta_+^2P_{m-2}(\Hh),
  \qquad \eta_+^2=|z|^2+4t,
\end{equation}
where $P_j(\Hh)=\{0\}$ for $j<0$.
Equivalently, every $p\in P_m(\Hh)$ has a unique representation
\[
  p=h+\eta_+^2q,
  \qquad h\in H_m(\Hh),\quad q\in P_{m-2}(\Hh).
\]
\end{theorem}

\begin{proof}
Let $m\geq2$ and $p\in P_m(\Hh)$.
By~\cref{lem:Fischer-auto}, there is a unique polynomial $q\in P_{m-2}(\Hh)$ such that
\[
  A_{m-2}q=\Delta_Hp.
\]
Then $h=p-wq$ is homogeneous of degree $m$ and satisfies $\Delta_Hh=0$, proving existence. If $wq$ is harmonic, then $A_{m-2}q=0$, so~\cref{lem:Fischer-auto} gives $q=0$. Thus the sum is direct and the representation is unique.
\end{proof}

\begin{remark}
The direct sum in~\eqref{eq:eta-direct} is algebraic, but it is \textbf{not} an
orthogonal decomposition of polynomial traces. For instance,
$t\in H_2(\Hh)$ and $w=\eta_+^2\in wP_0(\Hh)$, while
\eqref{eq:sphere-param} and~\eqref{eq:sigma-explicit} give
\[
  \langle t,w\rangle_\sigma
  =2\pi\int_0^\pi
      \cos\phi\bigl(\sin\phi+4\cos\phi\bigr)\,\dd\phi
  =4\pi^2\neq0.
\]
\end{remark}

\begin{corollary}\label{cor:dimensions}
For every $m\geq0$,
\[
  \dim H_m(\Hh)=m+1.
\]
More explicitly, for $k\geq0$,
\[
  \dim P_{2k}(\Hh)=(k+1)^2,
  \qquad
  \dim P_{2k+1}(\Hh)=(k+1)(k+2).
\]
\end{corollary}

\begin{proof}
The formulas for $P_m$ follow by counting the monomials $x^ay^bt^c$ such that $a+b+2c=m$. Finally, since multiplication by the nonzero polynomial $w$ is injective,~\eqref{eq:eta-direct} gives
\[
  \dim H_m(\Hh)=\dim P_m(\Hh)-\dim P_{m-2}(\Hh)=m+1.
\]
\end{proof}

\begin{corollary}\label{cor:iterated}
Every $p_m\in P_m(\Hh)$ has a unique expansion
\[
  p_m=\sum_{k=0}^{\lfloor m/2\rfloor}\eta_+^{2k}h_{m-2k},
  \qquad h_{m-2k}\in H_{m-2k}(\Hh).
\]
\end{corollary}

\begin{proof}
Apply~\cref{thm:eta-decomposition} successively to the lower-degree factor.
Uniqueness at every step gives uniqueness of the full expansion.
\end{proof}

\begin{remark}
For the heat operator $\mathcal H=\Delta_x-\partial_t$ and the degree-two polynomial
$v(x,t)=t-|x|^2$, Kogoj and Lanconelli~\cite[Thm.~1.1]{KogojLanconelli2022} prove that for every polynomial $p$ there is a unique polynomial $q$ satisfying
\[
  \mathcal H(vq)=-\mathcal Hp.
\]
Then $p+vq$ is caloric and agrees with $p$ on $\{v=0\}$.
\cref{prop:Dirichlet-correction} has the same degree-two correction mechanism under anisotropic scaling, but uses the Heisenberg sub-Laplacian and the paraboloid $\{w=1\}$.
After the translation $s=t-1/4$, one has
\[
1-w=-(|z|^2+4s), 
\]
which makes the common anisotropic degree-two mechanism
particularly transparent while leaving the operators and geometries
distinct.
\end{remark}

\begin{remark}
On a compact subset of a fixed level $\{w=c\}$, the iterated expansion in~\cref{cor:iterated} shows that polynomial restrictions lie in the span of harmonic-polynomial restrictions. Stone--Weierstrass therefore gives uniform density on such a compact set. The full level $\{w=c\}$ is unbounded, and this observation neither supplies a global $L^2$ theorem there nor implies an orthogonality statement. It also does not replace the independent completeness theorem on the Kor\'anyi sphere, where $w$ is not constant.
\end{remark}

\section{Concluding remarks}

The Heisenberg spherical picture has two complementary forms. Homogeneous
harmonic traces are algebraically direct and complete, and inequivalent
angular types are orthogonal; homogeneous degrees are not. The
degree-one/degree-three obstruction proves that no finite positive
full-support surface measure can restore raw degree orthogonality on any
$\Hh^n$.
Orthogonalizing the cumulative harmonic spaces instead gives
\eqref{eq:Hn-corrected-Hilbert-sum}, with the same dimensions and $U(n)$
types as the original homogeneous spaces.
\ref{sec:Hn-harmonic-basis} supplies an explicit
triangular basis construction behind those dimensions.

On a general Carnot group, the harmonic construction is complete on its
closed trace span and becomes a decomposition of all spherical $L^2$ exactly
under the separate condition $(\mathrm{HD})_N$. The
filtration~\eqref{eq:carnot-all-polynomial-sum}, on the other hand, is unconditional.

The polynomial factorization~\eqref{eq:eta-direct} is separate from these
spherical constructions. Its proof rests on the special polynomial
$w=|z|^2+4t$ and on the automorphisms
$A_j=\Delta_H\circ M_w$. Extending this Fischer mechanism to
higher-dimensional Heisenberg groups remains a natural question. Any
extension to more general Carnot groups would require a polynomial multiplier
and a new invertibility argument; it does not follow formally from the
existence of a homogeneous norm or a polar integration formula.

\section*{Acknowledgments}

The author thanks Prof.~Vladimir Georgiev (University of Pisa) for useful feedback and insightful remarks on an earlier draft. The author also thanks the anonymous referee for the suggestions that led to a clearer separation of the spherical completeness result from the algebraic decomposition.

\appendix
\renewcommand{\thetheorem}{\Alph{section}.\arabic{theorem}}
\renewcommand{\thelemma}{\Alph{section}.\arabic{theorem}}
\renewcommand{\thecorollary}{\Alph{section}.\arabic{theorem}}
\renewcommand{\theproposition}{\Alph{section}.\arabic{theorem}}
\renewcommand{\thedefinition}{\Alph{section}.\arabic{theorem}}
\renewcommand{\theremark}{\Alph{section}.\arabic{theorem}}
\renewcommand{\theexample}{\Alph{section}.\arabic{theorem}}

\section{Constructive harmonic basis on \texorpdfstring{$\Hh^n$}{Hn}}
\label{sec:Hn-harmonic-basis}

The harmonic-polynomial construction extends cleanly to every
Heisenberg group, independently of the orthogonality question for traces on a
Kor\'anyi sphere. This section gives a self-contained construction. Besides producing bases, the argument proves directly that the sub-Laplacian is onto
between consecutive weighted homogeneous polynomial spaces. Let
\[
  \Hh^n=\C^n\times\R,
  \qquad z_j=x_j+\imath y_j,
\]
with group law
\[
  (z,t)(z',t')=
  \left(z+z',t+t'+2\operatorname{Im}
  \sum_{j=1}^n z_j\overline{z'_j}\right).
\]
The horizontal vector fields and sub-Laplacian are
\[
  X_j=\partial_{x_j}+2y_j\partial_t,
  \qquad
  Y_j=\partial_{y_j}-2x_j\partial_t,
  \qquad
  \Delta_H=\sum_{j=1}^n(X_j^2+Y_j^2).
\]
Set $D=2n$, and
\[
    r^2=|z|^2=\sum_{j=1}^n(x_j^2+y_j^2),
  \qquad
  \mathcal R=\sum_{j=1}^n
  \bigl(y_j\partial_{x_j}-x_j\partial_{y_j}\bigr).
\]
As mentioned already, a direct expansion gives
\begin{equation}\label{eq:Hn-expanded-Delta}
  \Delta_H=\Delta_z+4\mathcal R\,\partial_t
  +4r^2\partial_t^2,
  \qquad
  \Delta_z=\sum_{j=1}^n
  (\partial_{x_j}^2+\partial_{y_j}^2).
\end{equation}
Horizontal variables have weight one, while $t$ has weight two. Let
$P_m(\Hh^n)$ be the complex polynomials of weighted homogeneous degree $m$,
and
\[
  H_m(\Hh^n)=\ker\bigl(\Delta_H:P_m(\Hh^n)
  \longrightarrow P_{m-2}(\Hh^n)\bigr).
\]
As usual, a polynomial space with a negative subscript is understood to be
equal to $\{0\}$.

For the horizontal variables alone, denote by $\mathscr P_d(\R^D)$ the
ordinary homogeneous polynomials of degree $d$, and by
\[
  \mathscr Y_d(\R^D)
  =\ker\bigl(\Delta_z:\mathscr P_d(\R^D)
  \longrightarrow\mathscr P_{d-2}(\R^D)\bigr)
\]
the Euclidean harmonic polynomials.

\begin{lemma}
\label{lem:Euclidean-right-inverse}
For $d\geq0$, define $\mathcal G_d$ as follows:
\begin{equation}\label{eq:Euclidean-right-inverse}
  \mathcal G_dq
  =\sum_{j=0}^{\lfloor d/2\rfloor}
  \frac{(-1)^j r^{2j+2}\Delta_z^jq}
  {\displaystyle\prod_{\ell=0}^{j}
  2(\ell+1)(2d+D-2\ell)},
  \qquad q\in\mathscr P_d(\R^D).
\end{equation}
Then $\mathcal G_dq\in\mathscr P_{d+2}(\R^D)$ and
\[
  \Delta_z\mathcal G_dq=q.
\]
In particular,
$\Delta_z:\mathscr P_{d+2}(\R^D)\to\mathscr P_d(\R^D)$ is surjective.
\end{lemma}

\begin{proof}
If $q$ is homogeneous of degree $d$, the Euler identity gives, for $a\geq1$,
\begin{equation}\label{eq:radial-Laplacian-identity}
  \Delta_z(r^{2a}q)
  =r^{2a}\Delta_zq
  +2a(2d+D+2a-2)r^{2a-2}q.
\end{equation}
Apply this identity to $\Delta_z^jq$, which has degree $d-2j$, with
$a=j+1$. A simple computation shows that the coefficient of $r^{2j}\Delta_z^jq$ is
\[
  \alpha_j=2(j+1)(2d+D-2j)>0.
\]
In the Laplacian of the sum in~\eqref{eq:Euclidean-right-inverse}, the
$r^{2j}\Delta_z^jq$ term contributed by the radial part of the $j$th
summand cancels the term contributed by the Laplacian part of the
$(j-1)$st summand. Hence only the radial part of the zeroth summand remains,
and it equals $q$. The terminal Laplacian term is zero because
$\Delta_z^{\lfloor d/2\rfloor+1}q=0$. This proves the identity.
\end{proof}

Every polynomial $p\in P_m(\Hh^n)$ has a unique $t$-expansion
\begin{equation}\label{eq:t-expansion-Hn}
  p(z,t)=\sum_{k=0}^{M}t^ku_k(z),
  \qquad
  M=\left\lfloor\frac m2\right\rfloor,
  \qquad
  u_k\in\mathscr P_{m-2k}(\R^D).
\end{equation}
If $u_k=0$ outside the index range $0\leq k\leq M$, then~\eqref{eq:Hn-expanded-Delta}
shows that the coefficient of $t^k$ in $\Delta_Hp$ is
\begin{equation}\label{eq:triangular-coefficient}
  \Delta_zu_k
  +4(k+1)\mathcal Ru_{k+1}
  +4(k+2)(k+1)r^2u_{k+2}.
\end{equation}
Notice that this is triangular when the coefficients are read from the highest power of
$t$ downwards.

\begin{proposition}
\label{prop:Hn-surjectivity}
For every $m\geq2$, the map
\[
  \Delta_H:P_m(\Hh^n)\longrightarrow P_{m-2}(\Hh^n)
\]
is surjective.
\end{proposition}

\begin{proof}
Let $M=\lfloor m/2\rfloor$ and write an arbitrary target as
\[
  f(z,t)=\sum_{k=0}^{M-1}t^kf_k(z),
  \qquad f_k\in\mathscr P_{m-2-2k}(\R^D).
\]
We construct $p$ as in~\eqref{eq:t-expansion-Hn}. Set
$u_M=u_{M+1}=0$. For $k=M-1,M-2,\ldots,0$, suppose that
$u_{k+1}$ and $u_{k+2}$ have already been defined and set
\begin{align}
  g_k={}&f_k-4(k+1)\mathcal Ru_{k+1}
  -4(k+2)(k+1)r^2u_{k+2},\label{eq:surjectivity-gk}\\[.4em]
  u_k={}&\mathcal G_{m-2-2k}g_k.\label{eq:surjectivity-uk}
\end{align}
All degrees in these formulas match. Indeed, $g_k$ has horizontal degree
$m-2-2k$, and $u_k$ has horizontal degree $m-2k$. By
\cref{lem:Euclidean-right-inverse}, 
\[
\Delta_zu_k=g_k.
\]
Substitution in
\eqref{eq:triangular-coefficient} therefore makes the coefficient of $t^k$
equal to $f_k$ for each $0\leq k\leq M-1$.
There is no additional
coefficient at $t^M$ since the horizontal coefficient there would have degree
$m-2M\in\{0,1\}$ and hence has zero Euclidean Laplacian. Thus
$\Delta_Hp=f$.
\end{proof}

The triangular system also parametrizes every harmonic polynomial by
ordinary Euclidean harmonic data.

\begin{theorem}[Triangular construction]
\label{thm:Hn-triangular-basis}
Fix $m\geq0$, put $M=\lfloor m/2\rfloor$, and choose any
\[
  v_k\in\mathscr Y_{m-2k}(\R^D),
  \qquad 0\leq k\leq M.
\]
Set $u_{M+1}=u_{M+2}=0$ and $u_M=v_M$.
Recursively, for
$k=M-1,M-2,\ldots,0$, define
\begin{align}
  b_k={}&-4(k+1)\mathcal Ru_{k+1}
  -4(k+2)(k+1)r^2u_{k+2},\label{eq:harmonic-bk}\\[.4em]
  u_k={}&v_k+\mathcal G_{m-2k-2}b_k.\label{eq:harmonic-uk}
\end{align}
Then
\begin{equation}\label{eq:harmonic-from-data}
  h(z,t)=\sum_{k=0}^{M}t^ku_k(z) \in H_m(\Hh^n).
\end{equation}
Moreover, every element of $H_m(\Hh^n)$ is
obtained in exactly one way. Consequently, the construction is a linear
isomorphism
\begin{equation}\label{eq:Hn-vector-space-isomorphism}
  \bigoplus_{k=0}^{\lfloor m/2\rfloor}
  \mathscr Y_{m-2k}(\R^{2n})
  \simeq H_m(\Hh^n).
\end{equation}
\end{theorem}

\begin{proof}
Because $\Delta_zv_k=0$ and
$\Delta_z\mathcal G_{m-2k-2}b_k=b_k$, equations
\eqref{eq:harmonic-bk}--\eqref{eq:harmonic-uk} give
\[
  \Delta_zu_k+4(k+1)\mathcal Ru_{k+1}
  +4(k+2)(k+1)r^2u_{k+2}=0.
\]
Together with~\eqref{eq:triangular-coefficient}, this proves
$\Delta_Hh=0$.

\smallskip
Conversely, let a harmonic $h$ have coefficients $u_k$ as in
\eqref{eq:t-expansion-Hn}. Starting with the highest coefficient and moving
downwards, define $b_k$ by~\eqref{eq:harmonic-bk} and
\[
  v_M=u_M,
  \qquad
  v_k=u_k-\mathcal G_{m-2k-2}b_k
  \quad(0\leq k<M).
\]
The harmonic coefficient equations imply $\Delta_zv_k=0$. Hence the
$v_k$ lie in the asserted Euclidean harmonic spaces and reconstruct the
original coefficients by~\eqref{eq:harmonic-uk}. This also proves
uniqueness.
\end{proof}

\begin{remark}
Choose any basis $\mathcal B_d$ of Euclidean harmonic spaces
$\mathscr Y_d(\R^{2n})$. In~\cref{thm:Hn-triangular-basis}, take one
$v_k$ at a time to be an element of $\mathcal B_{m-2k}$ and set all the
other $v_j$ equal to zero. The polynomials produced by the recursion form a
basis of $H_m(\Hh^n)$. Thus~\eqref{eq:harmonic-uk} is constructive once a Euclidean homogeneous harmonic-polynomial basis has been chosen.
\end{remark}

\begin{corollary}
\label{cor:Hn-harmonic-dimension}
For every $m\geq0$,
\begin{equation}\label{eq:Hn-dimension}
  \dim H_m(\Hh^n)=\binom{2n+m-1}{2n-1}.
\end{equation}
\end{corollary}

\begin{proof}
First,~\cref{lem:Euclidean-right-inverse} gives the exact sequence
\[
  0\longrightarrow\mathscr Y_d(\R^D)
  \longrightarrow\mathscr P_d(\R^D)
  \xrightarrow{\,\Delta_z\,}\mathscr P_{d-2}(\R^D)
  \longrightarrow0.
\]
Therefore, with binomial coefficients having a negative lower index
interpreted as zero, this yields
\[
  \dim\mathscr Y_d(\R^D)
  =\binom{D+d-1}{d}-\binom{D+d-3}{d-2}.
\]
Summing this identity over $d=m,m-2,\ldots$ in
\eqref{eq:Hn-vector-space-isomorphism} gives
\eqref{eq:Hn-dimension}.
\end{proof}

We finish by showing the first few bases explicitly in the normalization
used in this article. On $\Hh=\C\times\R$, if
$z=x+\imath y$, then
\[
  \Delta_H(z^a\bar z^bt^c)=
  4ab\,z^{a-1}\bar z^{b-1}t^c
  +4\imath c(b-a)z^a\bar z^bt^{c-1} +4c(c-1)z^{a+1}\bar z^{b+1}t^{c-2}.
\]
Terms carrying a zero prefactor are understood to be omitted, so no negative
power occurs when $a=0$, $b=0$, or $c<2$.
\begin{equation}\label{eq:H1-low-degree-bases}
\begin{array}{c@{\quad}|@{\quad}l}
 m & \text{a complex basis of }H_m(\Hh)\\[.4em]
0 & 1\\[2pt]
1 & z,\ \bar z\\[2pt]
2 & z^2,\ t,\ \bar z^2\\[2pt]
3 & z^3,\ z\left(t+\tfrac{\imath }2|z|^2\right),\
    \bar z\left(t-\tfrac{\imath}2|z|^2\right),\ \bar z^3\\[6pt]
4 & z^4,\ z^2\left(t+\tfrac{2\imath}3|z|^2\right),\
    t^2-\tfrac12|z|^4,\
    \bar z^2\left(t-\tfrac{2\imath}3|z|^2\right),\ \bar z^4.
\end{array}
\end{equation}
For each fixed $m$, the displayed polynomials have distinct angular
frequencies and hence are linearly independent. Their number is $m+1$, so
\cref{cor:Hn-harmonic-dimension} shows that each row is a basis.

\section{Invertibility of the Fischer operator}
\label{sec:fischer-appendix}

This appendix contains the technical finite-dimensional argument used in
the $\eta_+^2$-decomposition. Recall that, for $j\geq0$, the Fischer operator is
\[
  A_j:P_j(\Hh)\longrightarrow P_j(\Hh),
  \qquad
  A_jq=\Delta_H(wq),
\]
where $w=\eta_+^2$.

\begin{proof}[Proof of \cref{lem:Fischer-auto}]
We decompose $P_j(\Hh)$ under rotations. Set $r=|z|$ and
\[
  \tau=t+\tfrac{r^2}{4},
  \qquad w=4\tau.
\]
In the coordinates $(r,\theta,\tau)$, formula~\eqref{eq:Delta-polar-rt} becomes
\begin{equation}\label{eq:Delta-tau}
\Delta_H
=\partial_r^2+\frac1r\partial_r+\frac1{r^2}\partial_\theta^2
+r\partial_r\partial_\tau+(1-4\partial_\theta)\partial_\tau
+\frac{17}{4}r^2\partial_\tau^2.
\end{equation}
Although polar notation is used, all the spaces below consist of polynomials, so the identities extend across $r=0$. The allowed angular frequencies are
\[
  \Lambda_j=\bigl\{\nu\in\mathbb Z \ : \ |\nu|\leq j,\ j-|\nu|\text{ is even}\bigr\}.
\]
For $\nu\in\Lambda_j$ set $L=(j-|\nu|)/2$. The frequency-$\nu$
component has the polynomial basis
\[
  e_\ell^{(j,\nu)}
  =\tau^\ell r^{j-2\ell}e^{\mathrm{i}\nu\theta},
  \qquad 0\leq\ell\leq L.
\]
For example, when $\nu\geq0$ the horizontal factor is
$z^\nu(r^2)^{(j-2\ell-\nu)/2}$.
Substitution in~\eqref{eq:Delta-tau} gives
\begin{equation} \label{eq:A-block}
\begin{aligned}
A_je_\ell^{(j,\nu)} = {} &
17\ell(\ell+1)e_{\ell-1}^{(j,\nu)}
+4(\ell+1)(j-2\ell+1-4\mathrm{i}\nu)e_\ell^{(j,\nu)} \\
&\qquad +4\bigl((j-2\ell)^2-\nu^2\bigr)e_{\ell+1}^{(j,\nu)}.
\end{aligned}
\end{equation}
Suppose first that $\nu\neq0$. For $0\leq\ell<L$, the product of the two
adjacent off-diagonal coefficients in~\eqref{eq:A-block} is
\[
  17(\ell+1)(\ell+2)\,
  4\bigl((j-2\ell)^2-\nu^2\bigr)>0.
\]
Hence any positive diagonal change of basis makes the off-diagonal part of the block matrix real symmetric. Its diagonal entries have imaginary parts equal to
$-16\nu(\ell+1)$. If the rescaled matrix $B$ satisfied $Bv=0$, then
\[
  0=\operatorname{Im}\langle Bv,v\rangle
  =-16\nu\sum_{\ell=0}^L(\ell+1)|v_\ell|^2,
\]
which forces $v=0$. Thus every nonzero angular block is invertible.

It remains to consider $\nu=0$, which occurs only for $j = 2N-2$ even. If $A_jq=0$ in the rotation-invariant block, then
$h=wq$ is a rotation-invariant harmonic polynomial of degree $2N$. Set
\[
  R=(r^4+t^2)^{1/2}=\rho^2.
\]
By restriction injectivity in~\cref{prop:carnot-harmonic-direct} and the
rotation-type statement in~\cref{prop:angular}, the rotation-invariant harmonic
space of degree $2N$ is one-dimensional and it is spanned by
\begin{equation}\label{eq:zonal}
  Z_N(r,t)=R^N P_N(t/R),
\end{equation}
where $P_N$ is the Legendre polynomial.
Indeed, with $u=r^2$ one has
\[
  \Delta_Hf(u,t)=4u\left(f_{uu}+\frac1u f_u+f_{tt}\right),
\]
and~\eqref{eq:zonal} is the usual axisymmetric solid harmonic in three dimensions.
It is a polynomial because $P_N$ has the parity of $N$.

If $h=cZ_N$ is divisible by $w$, then it vanishes at $(r^2,t)=(4,-1)$. At this point, we have $R=\sqrt{17}$, and hence
\[
  Z_N(2,-1)=17^{N/2}P_N(-1/\sqrt{17})=(-1)^NS_N,
  \quad
  S_N:=17^{N/2}P_N(1/\sqrt{17}).
\]
The Legendre generating function gives, as a formal power series,
\[
  \sum_{N=0}^\infty S_N\zeta^N
  =(1-2\zeta+17\zeta^2)^{-1/2}
  =\bigl((1-\zeta)^2+16\zeta^2\bigr)^{-1/2}.
\]
Expanding the last expression yields
\begin{equation}\label{eq:SN-formula}
  S_N=
  \sum_{k=0}^{\lfloor N/2\rfloor}
  (-1)^k4^k\binom{2k}{k}\binom{N}{2k}.
\end{equation}
The $k=0$ term is one, and every term with $k\geq1$ is divisible by eight. Hence
\[
  S_N\equiv1\pmod8 \implies S_N \neq 0.
\]
Therefore $c=0$ and $q=0$. The zero angular block is injective as well.

\smallskip
This concludes the proof that every block of $A_j$ is injective. Since $P_j(\Hh)$ is finite-dimensional, $A_j$ is an automorphism.
\end{proof}

\end{document}